\documentclass[12pt,twoside, final]{amsart}

\usepackage{amsmath,amsthm,amscd,amsfonts,amssymb,enumerate}
\usepackage{graphicx}
\usepackage{color}
\usepackage[colorlinks]{hyperref}
\usepackage{amsfonts,amssymb,amscd,amsmath,enumerate,url,verbatim}
 \usepackage[dvips]{epsfig}
 \usepackage[none]{hyphenat}
\usepackage{amsmath,amssymb,amsfonts,enumerate,amsthm}
 \usepackage{amsgen, amstext,amsbsy,amsopn, amsthm, amsfonts,amssymb,amscd,amsmat
 h,euscript,enumerate,url,verbatim,calc,xypic}
 \usepackage{latexsym}
 \usepackage{graphics}
 \usepackage{color}
\numberwithin{equation}{section}

\numberwithin{equation}{section}
\newtheorem{thm}{Theorem}[section]
\newtheorem{lem}[thm]{Lemma}
\newtheorem{cor}[thm]{Corollary}

\newcommand{\Hom}{\mbox{Hom}\,}
\newcommand{\Ext}{\mbox{Ext}\,}

\newcommand{\Spec}{\mbox{Spec}\,}

\newcommand{\Ker}{\mbox{Ker}\,}
\newcommand{\Ass}{\mbox{Ass}}

\newcommand{\Supp}{\mbox{Supp}\,}

\newcommand{\grade}{\mbox{grade}\,}
\renewcommand{\Im}{\mbox{Im}\,}

\newcommand{\N}{\mathbb{N}}

\newcommand{\fa}{\mathfrak{a}}

\newcommand{\fp}{\mathfrak{p}}

\begin{document}

\title[On the associated prime ideals of local cohomology modules ...]
{On the associated prime ideals of local cohomology modules defined
by a pair of ideals}

\author[Kh. Ahmadi Amoli]{Khadijeh Ahmadi Amoli}
\address[Khadijeh Ahmadi Amoli]{Payame Noor University,  Po Box 19395-3697,  Tehran, Iran. }
 \email{khahmadi@pnu.ac.ir}

\author[Z. Habibi]{Zohreh Habibi}
\address[Zohreh Habibi]{  Payame Noor University,  Po Box 19395-3697 Tehran,
Iran.} \email{z\_habibi@pnu.ac.ir}

\author[M. Jahangiri]{Maryam Jahangiri }
\address[Maryam Jahangiri]{Faculty of Mathematical Sciences and Computer,  Kharazmi
University,  Tehran, Iran  and School of Mathematics,  Institute for
Research in Fundamental Sciences (IPM),  P.O. Box: 19395-5746,
Tehran, Iran. } \email{mjahangiri@ipm.ir and
jahangiri.maryam@gmail.com}

\thanks{The third author was in part supported by a grant from IPM (No.
92130111)}
%------------------------------------------------------------------------------------%

\maketitle

\begin{abstract}
Let $I$ and $J$ be two ideals of a commutative Noetherian ring $R$
and $M$ be an $R$-module. For a non-negative integer $n$ it is shown
that, if the sets $\Ass_R(\Ext^{n} _{R}(R/I,M))$ and
 $\Supp_R(\Ext^{i}_{R}(R/I,H^{j}_{I,J} (M)))$ are finite for all $i
\leq n+1$ and all $j< n$, then
 so is \linebreak$\Ass_R(\Hom_{R}(R/I,H^{n}_{I,J}(M)))$. We also study
the finiteness of $\Ass_R(\Ext^{i}_{R}(R/I,H^{n}_{I,J} (M)))$ for
$i=1,2$.
\end{abstract}

\section{\bf Introduction}

Let  $R$ be a commutative Noetherian ring, $I$ and $J$ be two ideals
of $R$ and $M$ be an $R$-module. For all $i\in \N_0$ the $i$-th
local cohomology functor with respect to $(I,J)$, denoted by
$H^{i}_{I,J}(-)$, defined by Takahashi et. all in \cite{TAK} as the
$i$-th right derived functor of the $(I,J)$- torsion functor $\Gamma
_{I,J}(-)$, where $$\Gamma _{I,J}(M):=\{x \in M : I^{n}x\subseteq Jx
\  \text {for} \   n\gg 1\}.$$ This notion coincides  with  the
ordinary local cohomology functor $H^{i}_{I }(-)$ when $J=0$, see
\cite{B-SH}.

The main motivation for this generalization comes from the study of
a dual of ordinary local cohomology modules $H^{i}_{I }(M)$
(\cite{sch}). Basic facts and more information about local
cohomology defined by a pair of ideals can be obtained from
\cite{TAK}, \cite{chu1} and \cite{chu2}.

Hartshorne in \cite{HART} proposed the following conjecture:

`` Let $M$ be a finitely generated $R$-module and $\fa$ be an ideal
of $R$. Then $\Ext^ {i}_{R}(R/\fa, H^{j}_{\fa}(M))$ is finitely
generated for all $i\geq 0$ and $j\geq 0$."

Also, Huneke in \cite{HU} raised some crucial problems on local
cohomology modules. One of them was about the finiteness of the set
of associated prime ideals of the local cohomology modules $H^{i}_{I
}(M)$.

Although there are some counterexamples to theses conjectures, see
\cite{si}, but there are some partial positive answers in some
special cases too, see for example \cite{b} or \cite{bl}.

In this paper, we consider these two problems for local cohomology
modules defined by a pair of ideals over not necessary finitely
generated modules. In particular,  we investigate certain conditions
on these modules such that the set of
 associated prime ideals of \linebreak$\Ext^{i}_{R}(R/I, H^{j}_{I,J} (M))$ is finite.

More precisely, let $n\in \N_0$ and assume that the sets
$\Ass_R(\Ext^{n} _{R}(R/I,M))$ and
\linebreak$\Supp_R(\Ext^{i}_{R}(R/I,H^{j}_{I,J} (M)))$ are finite
for all $i \leq n+ 1$ and all $j< n$ then, we use a spectral
sequence argument to show that
$\Ass_R(\Hom_{R}(R/I,H^{n}_{I,J}(M)))$ is finite, too (Theorem
\ref{ass1}).  Moreover, it is shown that if the sets
$\Ass_R(\Ext^{n+ 1} _{R}(R/I,M))$ and
 $\Supp(\Ext^{i}_{R}(R/I,H^{j}_{I,J} (M)))$ are finite for all $i
\leq n+ 2$ and all $j< n$ then, so is
 $\Ass_R(\Ext^{1}_{R}(R/I,H^{n}_{I,J} (M)))$
 (Theorem \ref{ass2}).

We also present a necessary and sufficient condition for the
finiteness of the set \linebreak$\Ass_R(\Ext^{2}_{R}(R/I,H^{n}_{I,J}
(M)))$ (Theorem \ref{ass3}). These generalize some known results
concerning ordinary local cohomology modules.

 In \cite[3.6]{T-T} the authors study the $\underset{M}{\grade\fp}$ for all $\fp\in \Ass_R (H^{t}_{I,J}(M))$, where
 $$t= inf\{i\in \mathbb{N}_{0}:
H^{i}_{I,J}(M)\neq 0 \}$$
 and $M$ is a finitely generated $R$-module. But their proof is not correct. Actually, they use the equality
 $\Supp_{R}(M_{x})= \{\fp \in \Supp_{R}(M): x \notin  \fp\}$ which is not true. Here, we also made a correction to this
 result for not necessary finite modules
 (Theorem \ref{tal}).

\section{ Associated prime ideals}%%%%%%%%%%%%%%%%%%%%%%%%%%%%%%%%%%%%%%%%%%%%%%%%%%%%%%%%%%%%%%%%%%%%%%%%%%%%%%%%%%%%%%%%%%%%%%%%%%%%%%%%%%%%%%%%%%%%%%%%

In this section, first, we are going to study the set of associated
prime ideals of some $\Ext$-modules of local cohomology modules
defined by a pair of ideals.

The following relation between associated prime ideals of modules in
an exact sequence is frequently used in our results.

\begin{lem}\label{*}
Let $M\rightarrow N\rightarrow K\rightarrow 0$ be an exact sequence
of $R$-modules. Then \emph{$\Ass(K)\subseteq \Supp(M)\cup \Ass(N)$.}
\end{lem}
\begin{proof}
Let $\fp\in \Ass(K)$. Assume that $ \fp \notin \Supp(M)$. Then
$M_{\fp}=0$ and so $N_{\fp}\cong K_{\fp}$. Since $\fp R_{\fp} \in
\Ass_{R_{\fp}}(N_{\fp})$, we get $\fp\in \Ass(N)$.
\end{proof}

Next lemma describes a convergence of Grothendieck spectral
sequences.
\begin{lem}\label{**}

Let $M$ be an $R$-module. Then the following convergence of spectral
sequences exists
$$\Ext^{i}_{R}(R/I,H^{j}_{I,J}(M))\overset{i}{\Rightarrow}\Ext^{i+j}
_{R}(R/I,M).$$

\begin{proof}
It is easy to see that $\Hom_{R}(R/I,\Gamma_{I,J}
(M))=\Hom_{R}(R/I,M)$. Also, for any injective $R$-module $E$,
$\Gamma_{I,J}(E)$ is an injective $R$-module, by \cite[3.2]{TAK} and
\cite[2.1.4]{B-SH}. Now, in view of \cite[10.47]{ROT}, the assertion
follows.
\end{proof}

\end{lem}

The following theorem, which concerns with Hartshorne's problem
mentioned in the introduction, is one of the main results in this
paper.
\begin{thm}\label{ass1}
Let $n$ be a non-negative integer and $M$ be an $R$-module such that
\linebreak \emph{$\Ass_R(\Ext^{n} _{R}(R/I,M))$} and
\emph{$\Supp_R(\Ext^{i}_{R}(R/I,H^{j}_{I,J}(M)))$} are finite for
all $i \leq n+1$ and all $j< n$. Then
 so is $\Ass_R(\Hom_{R}(R/I,H^{n}_{I,J}(M))).$
\end{thm}
\begin{proof}
Consider the convergence of spectral sequences in Lemma \ref{**} and
note that $E_{2}^{i,j}=0$ for all $i < 0$. Therefore, for all $2\leq
r \leq n+1$ there exists an exact sequence
\begin{equation}\label{r}
0\rightarrow E_{r+1}^{0,n}\rightarrow
E_{r}^{0,n}\xrightarrow{d_{r}^{0,n}}E_{r}^{r,n+1-r}.
\end{equation}

Since, $E_{r}^{r,n+1-r}$ is a subquotient of $E_{2}^{r,n+1-r}=
\Ext^{r}_{R}(R/I,H^{n+1-r}_{I,J} (M))$, $\Supp_R
 (E_{r}^{r,n+1-r})$ is a finite set. So, the above exact sequence  implies that $\sharp \Ass_R (E_{r}^{0,n})< \infty$ if $\sharp
\Ass_R(E_{r+1}^{0,n})< \infty$. Also, from the fact that
$E_{2}^{i,j}=0$ for all $j < 0$, we have $E_{\infty}^{0,n}\cong
E_{n+2}^{0,n}$. Therefore, to prove the assertion it is enough to
show that $\Ass_R(E_{\infty}^{0,n})$ is a finite set.

Using the concept of the convergence of spectral sequences, there
exists a bounded filtration
$$0=\varphi^{n+1}H^{n}\subseteq \varphi^{n}H^{n}\subseteq...
\subseteq \varphi^{1}H^{n}\subseteq\varphi^{0}H^{n}=
\Ext^{n}_{R}(R/I,M)$$ of submodules of $\Ext^{n}_{R}(R/I,M)$ such
that
$$E_{\infty}^{i,n-i}\cong \varphi^{i}H^{n}/\varphi^{i+1}H^{n} \
\text{for  all} \   i=0,...,n.$$
 Therefore, $E_{n+1}^{n,0}\cong
E_{\infty}^{n,0}\cong \varphi^{n}H^{n}$ is a subquotient of
$E_{2}^{n,0}=\Ext^{n}_{R}(R/I,\Gamma_{I,J}(M))$. So, by assumption,
$\Supp_R(\varphi^{n}H^{n})$ is a finite set. Now, assume inductively
that $\sharp \Supp_R(\varphi^{i}H^{n})< \infty$ for all $1<i\leq n$.
Then, since
$$E_{n+1}^{1,n-1}\cong
E_{\infty}^{1,n-1}\cong \varphi^{1}H^{n}/\varphi^{2}H^{n}$$
 is a
subquotient of $E_{2}^{1,n-1}=\Ext^{1}_{R}(R/I,H^{n-1}_{I,J}(M))$,
we deduce that $\Supp_R(\varphi^{1}H^{n})$ is finite. But,
$$E_{\infty}^{0,n}\cong \Ext^{n}_{R}(R/I,M)/\varphi^{1}H^{n}$$
 and
Lemma \ref{*} implies that $\sharp \Ass_R(E_{\infty}^{0,n})<\infty$,
as desired.

\end{proof}

As an immediate consequence of Theorem \ref{ass1}, we obtain the
following result that is a generalization of \cite[2.3]{B-S-SH}.
 \begin{cor}\label{max}
Let M be a finite R-module. Suppose that there is an integer n such
that for all $i < n$ the set \emph{$\Supp_R(H^{i}_{I,J}(M))$} is
finite. Then \emph{$\Ass_R(H^{n}_{I,J}(M))$} is finite.
\end{cor}

\begin{cor}\label{final}
Let $M$ be a finite $R$-module and
\emph{$t=inf\{i|H^{i}_{I,J}(M)\neq 0 \}$} be an   integer. Then
\emph{$\Ass_R(\Hom_{R}(R/I,H^{t}_{I,J}(M)))$} is finite. If in
addition, \emph{${\underset{M}{\grade I}}=t$}, then for a maximal
$M$-sequence $x_{1},...,x_{t}$ in $I$, we have
\emph{$$\Ass_R(\Hom_{R}(R/I,H^{t}_{I,J}(M)))= \{\fp\in \Ass_R(M/
(x_{1},..,x_{t})M)\cap V(I); \underset{M}{\grade\fp=}t\}.$$}
\end{cor}
\begin{proof}
It is straightforward from Theorem \ref{ass1}, \cite[3.10]{T-T} and
\cite[2.6]{AGH-AM}.
\end{proof}
\begin{cor}\label{max}
Let $M$ be a finite $R$-module. Suppose that   $q=inf \{i:
H^{i}_{I,J} (M)$ is not Artinian $\}$ is an  integer, then
\emph{$\Ass_R (\Hom_{R}(R/I,H^{q}_{I,J} (M)))$} is finite.
\end{cor}

In the rest of this paper we consider the set of associated prime
ideals of some $\Ext$ modules of local cohomology modules defined by
a pair of ideals.
\begin{thm}\label{ass2}
Let $n$ be a non-negative integer and $M$ be an $R$-module such that
\linebreak\emph{$\Ass_R(\Ext^{n+1} _{R}(R/I,M))$} and \emph{$\Supp_R
(\Ext^{i}_{R}(R/I,H^{j}_{I,J} (M)))$} are finite for all $i \leq
n+2$ and all $j<n$. Then so is  $\Ass_R(\Ext^{1}
_{R}(R/I,H^{n}_{I,J}(M)))$.
\end{thm}
\begin{proof}

Considering the convergence of the spectral sequences of Lemma
\ref{**}, we have to show that $\Ass_R(E_{2}^{1,n})$ is a finite
set. Using similar arguments as used in Theorem \ref{ass1}, one can
see that it is enough to show that $\Ass_R(E_{\infty}^{1,n})=
\Ass_R(E_{n+2}^{1,n})$ is a finite set.

By the concept of convergence of spectral sequences, there exists a
filtration
$$0=\varphi^{n+2}H^{n+1}\subseteq \varphi^{n+1}H^{n+1}\subseteq...
\subseteq \varphi^{1}H^{n+1}\subseteq\varphi^{0}H^{n+1}=
\Ext^{n+1}_{R}(R/I,M)$$ of submodules of $\Ext^{n+1}_{R}(R/I,M)$
such that $E_{\infty}^{i,n+1-i}\cong
\varphi^{i}H^{n+1}/\varphi^{i+1}H^{n+1}$ for all $i=0,...,n+1$.
Using the fact that $\sharp \Supp_R(E_{2}^{i,j})< \infty$ for all
$i\leq n+2$ and all $j<n$ one can see that
$\Supp_R(\varphi^{i}H^{n+1})$ is a finite set for all $i=2,...,n+2$.
Also, $\sharp \Ass_R(\varphi^{1}H^{n+1})<\infty$. Now, since
$$E_{n+2}^{1,n}\cong E_{\infty}^{1,n}\cong
\varphi^{1}H^{n+1}/\varphi^{2}H^{n+1},$$ using Lemma \ref{*}, we
have $\sharp \Ass_R(E_{\infty}^{1,n})<\infty$, and the result
follows.
\end{proof}

The following theorem presents a necessary and sufficient condition
for the finiteness of the set
$\Ass_R(\Ext^{i}_{R}(R/I,H^{n}_{I,J}(M)))$ when $ i= 1, 2.$

\begin{thm}\label{ass3}
Let $n$ be a non-negative integer and $M$ be an $R$-module such that
the sets \linebreak$\Supp_R(\Ext^{n+1} _{R}(R/I,M))$    and
 $\Supp_R(\Ext^{i}_{R}(R/I,H^{j}_{I,J} (M)))$  are finite for
all $i \leq n+2$ and all $j< n$. Then   $\Ass_R
(\Hom_{R}(R/I,H^{n+1}_{I,J}(M)))$ is finite if and only if
 $\Ass_R (\Ext^{2}_{R}(R/I,H^{n}_{I,J}(M)))$  is finite.

\end{thm}
\begin{proof}
$(\Leftarrow)$ Again, consider the convergence of spectral sequences
of Lemma \ref{**} and assume that $Ass_R (E_{2}^{2,n})$ is finite.
Since $E_{2}^{i,j}=0$ for all $i < 0$ or $j < 0$, using similar
arguments as used in Theorem \ref{ass1}, one can see that
$E_{\infty}^{0,n+1}\cong E_{n+3}^{0,n+1}$ and in order to prove that
$\sharp \Ass_R (E_{2}^{0,n+1})<\infty$ we have to show that $\sharp
\Ass_R(E_{\infty}^{0,n+1})<\infty$.

 There exists a filtration
$$0=\varphi^{n+2}H^{n+1}\subseteq \varphi^{n+1}H^{n+1}\subseteq...
\subseteq \varphi^{1}H^{n+1}\subseteq\varphi^{0}H^{n+1}=
\Ext^{n+1}_{R}(R/I,M)$$ of submodules of $\Ext^{n+1}_{R}(R/I,M)$
such that $E_{\infty}^{0,n+1}\cong
\Ext^{n+1}_{R}(R/I,M)/\varphi^{1}H^{n+1}$.
 Since \linebreak$\sharp \Ass_R(\Ext^{n+1}_{R}(R/I,M))<\infty $ we have $\sharp \Ass_R
(E_{\infty}^{0,n+1})<\infty$, as desired

$(\Rightarrow)$ Now, assume that $\Ass_R
(\Hom_{R}(R/I,H^{n+1}_{I,J}(M)))<\infty$ and consider the exact
sequence
$$0\rightarrow \Ker d_{2}^{0,n+1}\rightarrow
E_{2}^{0,n+1}\xrightarrow{d_{2}^{0,n+1}} \Im
d_{2}^{0,n+1}\rightarrow 0.$$ Since $\Ker
d_{2}^{0,n+1}=E_{3}^{0,n+1}$ and $\sharp \Supp_R (E_{3}^{0,n+1})<
\infty$, in view of Lemma \ref{*}, we have \linebreak $\sharp \Ass_R
(\Im d_{2 }^{0,n+1})< \infty$. Now, using the exact sequence
$$0\rightarrow \Im d_{2}^{0,n+1}\rightarrow
E_{2}^{2,n}\xrightarrow{d_{2}^{2,n}} E_{2}^{4,n-1}$$ and the fact
that $ E_{2}^{4,n-1}=\Ext^{4}_{R}(R/I,H^{n-1}_{I,J}(M))$ has finite
support, we have $\sharp \Ass_R (E_{2}^{2,n})<\infty$, as desired.

\end{proof}
\begin{thm}\label{assnd}
Let $n$ be a non-negative integer and $M$ be an $R$-module of
dimension $d$, such that \emph{$\Ass_R(\Ext^{n+d} _{R}(R/I,M))$} and
\emph{$\Supp_R (\Ext^{i}_{R}(R/I,H^{j}_{I,J} (M)))$} are finite for
all $i \geq n+1$ and all $j<d$. Then \emph{$\Ass_R(\Ext^{n}
_{R}(R/I,H^{d}_{I,J}(M)))$} is finite.
\end{thm}
\begin{proof}
The method of the proof is similar to the Theorem \ref{ass2},
considering \cite[4.7]{TAK}.

%Considering the spectral sequences of Lemma \ref{**} we have to show
%that $\Ass(E_{2}^{n,d})$ is a finite set. using similar arguments as
%used in Theorem \ref{ass1}, one can see that it is enough to show
%that there exists an integer $k$ such that
%$\Ass(E_{\infty}^{n,d})\cong \Ass(E_{k}^{n,d})$ is finite. There
%exists a filtration

%$0=\varphi^{n+d+1}H^{n+d+1}\subseteq
%\varphi^{n+d}H^{n+d}\subseteq... \subseteq
%\varphi^{1}H^{n+d}\subseteq\varphi^{0}H^{n+d}=
%\Ext^{n+d}_{R}(R/I,M)$ of submodules of $\Ext^{n+d}_{R}(R/I,M)$ such
%that $E_{\infty}^{i,n+d-i}\cong
%\varphi^{i}H^{n+d}/\varphi^{i+1}H^{n+d}$ for all $i=0,...,n+d$.
%Using the fact that $\sharp \Supp(E_{2}^{i,j})< \infty$ for all
%$i\geq n+2$, one can see that $\Supp(\varphi^{i}H^{n+d})$ is a
%finite set for all $i=n+1,...,n+d$. Also, $\sharp
%\Ass(\varphi^{n}H^{n+d})<\infty$. Now, since $E_{k}^{n,d}\cong
%E_{\infty}^{n,d}\cong \varphi^{n}H^{n+d}/\varphi^{n+1}H^{n+d}$,
%using Lemma \ref{*}, we have $\sharp \Ass(E_{\infty}^{n,d})<\infty$,
%and the result follows.
\end{proof}

In the rest of this paper, we study "the grade" of prime ideals
$\fp\in \Ass_R (H^{t}_{I,J}(M))$ on $M$.

For an $R$-module $M$ and an ideal $\fa$ of $R$,   the grade of
$\fa$ on $M$ is defined by
$$\underset{M}{\grade\fa} := \inf \{i\in
\mathbb{N}_{0} : H^{i}_{\fa}(M)\neq 0\},$$
 if this infimum exists,
and $\infty$ otherwise. If $M$ is a finite $R$-module and $\fa M\neq
M$, this definition coincides with the length of a maximal
$M$-sequence in $\fa$ (cf. \cite[6.2.7]{B-SH}).

Also, we shall use the following notations introduced in \cite{TAK},
in which $W(I,J)$ is closed under specialization, but not
necessarily a closed subset of $\Spec(R)$. \[W(I,J):= \{\fp\in
\Spec(R): I^{n}\subseteq \fp + J \mbox{ for some integer }
n\geq1\},\] and

\[\widetilde{W}(I,J):= \{\fa: \fa \mbox{  is an ideal of } R \mbox{and} I^{n}\subseteq
\fa+J \mbox{ for some integer } n\geq 1\}.\]

The following lemma can be proved using \cite[3.2]{TAK}.

\begin{lem}\label{supp}
For any non-negative integer $i$ and $R$-module $M$,

$(i)$ \emph{$\Supp_R(H^{i}_{I,J}(M))\subseteq \underset{\fa \in
\widetilde{W}(I,J)}{\bigcup}\Supp(H^{i}_{\fa}(M))$}.

$(ii)$ \emph{$\Supp_R(H^{i}_{I,J}(M))\subseteq \Supp_R(M) \cap
W(I,J)$}.

\end{lem}

 The following theorem was proved in \cite[3.6]{T-T} under the hypothesis that $M$ is finite. But the proof is
not correct. Here we bring an extension and another proof of this
theorem.
\begin{thm}\label{tal}
Let $M$ be an $R$-module and $t=inf\{i\in \mathbb{N}_{0}:
H^{i}_{I,J}(M)\neq 0 \}$ be a non-negative integer. Then for all
\emph{$\fp\in \Ass_R (H^{t}_{I,J}(M))$},
\emph{$\underset{M}{\grade\fp=} t$.}
\end{thm}

\begin{proof}
We use induction on $t$. Let $t = 0$ and $\fp\in \Ass_R
(\Gamma_{I,J}(M))$. Then $\fp =(0:_{R}x)$ for some $x\in
\Gamma_{I,J}(M)$. Hence $x\in\Gamma_{\fp}(M)$ and so
$\Gamma_{\fp}(M)\neq 0$.

 Now suppose that $t>0$ and the case $t - 1$
is settled. Let $\fp \in \Ass(H^{t}_{I,J}(M))$ and consider the
exact sequence $0\rightarrow M\rightarrow E\rightarrow L\rightarrow
0$, where $E=E_{R}(M)$ is the injective envelope of $M$. Therefore,
using \cite[2.2]{T-T}, $H^{i}_{I,J}(L)\cong H^{i+1}_{I,J}(M)$ for
all $i\geq 0$ and we get
$$inf\{i\in \mathbb{N}_{0}: H^{i}_{I,J}(L)\neq 0\}= inf\{i\in
\mathbb{N}_{0}: H^{i}_{I,J}(M)\neq 0 \}-1=t-1$$ and that
$\fp\in\Ass_R(H^{t-1}_{I,J}(L))$ . Thus, by inductive hypothesis,
$\underset{L}{\grade\fp=} t-1$. Now, consider the long exact
sequence
$$H^{i-1}_{\fp}(M)\rightarrow H^{i-1}_{\fp}(E)\rightarrow
H^{i-1}_{\fp}(L)\rightarrow H^{i}_{\fp}(M).$$ If $t>1$, then
$H^{i}_{\fp}(M)\cong H^{i-1}_{\fp}(L)=0$ for all $i<t$ and
$H^{t}_{\fp}(M)\cong H^{t-1}_{\fp}(L)\neq 0$. Thus
$\underset{M}{\grade\fp=} t$.

Let $t=1$. Then $\Gamma_{\fp}(L)\neq 0$. By the above exact
sequence, it is enough to show that $\Gamma_{\fp}(E)=0$. On the
contrary, assume that $\Gamma_{\fp}(E)\neq 0$. Then there exists a
non-zero element $x\in E$ and $n\in \mathbb{N}$ such that
$\fp^{n}x=0$. We may assume that $\fp^{n}x=0$ and $\fp^{n-1}x\neq
0$. So, there exists $r\in \fp^{n-1}$ such that $rx\neq 0$. Thus
$\fp\subseteq (0:_{R}rx)$. On the other hand, by Lemma \ref{supp},
$$\fp\in \Ass_R(H^{1}_{I,J}(M))\subseteq \Supp_R
(H^{1}_{I,J}(M))\subseteq \underset{\fa \in
\widetilde{W}(I,J)}{\bigcup}\Supp_R(H^{1}_{\fa}(M)).$$ So that there
exists $\fa \in \widetilde{W}(I,J)$ such that $\fa\subseteq \fp$.
Let $m\in\mathbb{N}$ with $I^{m}\subseteq \fa +J\subseteq
\fp+J\subseteq (0:_{R}rx)+J$. Hence $rx\in \Gamma_{I,J}(M)$ which
contradicts with hypothesis and the choice of $rx$. Therefore
$\Gamma_{\fp}(E)=0$ and so $\underset{M}{\grade\fp=} 1$.
\end{proof}

% ------------------------------------------------------------------------

%\subsection*{Acknowledgment}

% ------------------------------------------------------------------------
\end{document}